\newif\ifpictures
\picturestrue

\newif\ifcomment
\commenttrue

\newif\ifQCQP
\QCQPfalse

\documentclass[12pt]{amsart}
\usepackage{latex_base}
\usepackage{macros}

\usepackage{footnote}
\makesavenoteenv{table}

\author{Mareike Dressler}
\address{Mareike Dressler, University of California, San Diego
	Department of Mathematics
	9500 Gilman Drive 0112
	La Jolla, CA 92093-0112
	USA\medskip}
\email{mdressler@ucsd.edu}

\author{Janin Heuer}
\address{Janin Heuer, Technische Universit\"at Braunschweig, Institut f\"ur Analysis und Algebra, AG Algebra, Universit\"atsplatz 2, 38106 Braunschweig,
 Germany\medskip}
\email{janin.heuer@tu-braunschweig.de}

\author{Helen Naumann}
\address{Helen Naumann, Goethe Universit\"at, FB 12- Institut f\"ur Mathematik, Postfach 11 19 31, D-60054, Frankfurt a.M.
	Germany\medskip}
\email{naumann@math.uni-frankfurt.de}

 \author{Timo de Wolff}
\address{Timo de Wolff, Technische Universit\"at Braunschweig, Institut f\"ur Analysis und Algebra, AG Algebra, Universit\"atsplatz 2, 38106 Braunschweig,
 Germany\medskip}
\email{t.de-wolff@tu-braunschweig.de}



\title[]{Global Optimization via the Dual SONC Cone and Linear Programming}

\begin{document}

\begin{abstract}
	Using the dual cone of sums of nonnegative circuits (SONC), we provide a relaxation of the global optimization problem to minimize an exponential sum and, as a special case, a multivariate real polynomial.
	Our approach builds on two key observations. 
	First, that the dual SONC cone is contained in the primal one.
	Hence, containment in this cone is a certificate of nonnegativity.
	Second, we show that membership in the dual cone can be verified by a linear program.
	We implement the algorithm and present initial experimental results comparing our method to existing approaches.

\end{abstract}

\maketitle

\section{Introduction}

Let $\struc{A} \subseteq \R^n$ be a finite set and let $\struc{\R^A}$ denote the space of all \textit{\struc{(sparse) exponential sums}} (signomials) supported on $A$. These are of the form
\begin{align}
	\struc{f}=\sum_{{\alpb} \in A} c_{{\alpb}} e^{\langle \xb,{{\alpb}}\rangle } \in\R^A, \quad c_{{\alpb}} \in \R \text{ for all } \alpb \in A. \label{setupgen}
\end{align}

We consider the following global optimization problem
\begin{align}
	\inf_{\xb \in \R^n} f(\xb),
	\label{Intro:KeyProblem}
\end{align}
which is the unconstrained version of a \textit{\struc{signomial optimization problem}}. Signomial programs are a rich class of nonconvex optimization problems with a broad range of applications; see e.g., \cite{Boyd:et:al:TutorialOnGP,Duffin:Peterson:Signomials} for an overview.

If $A \subseteq \N^n$, then $\R^A$ coincides with the space of real polynomials on the positive orthant supported on $A$. Thus, \cref{Intro:KeyProblem} also represents all unconstrained \textit{\struc{polynomial optimization problems}} on $\R_{>0}^n$; see e.g. \cite{Blekherman:Parrilo:Thomas, Lasserre:BookMomentsApplications, Lasserre:IntroductionPolynomialandSemiAlgebraicOptimization} for an overview about polynomial optimization problems and their applications.

\medskip

Under the assumption that \cref{Intro:KeyProblem} has a finite solution, minimizing $f \in \R^A$ is equivalent to adding a minimal constant $\gamma$ such that $f + \gamma \geq 0$. 
Hence, we consider the (convex, closed) \struc{\textit{sparse nonnegativity cone}} in $\R^A$, which is defined as
\begin{align}
	\struc{\NonnegCone} \ = \ \{f \in \R^A \ : \ f(\Vector{x}) \geq 0 \text{ for all } \xb \in \R^n\}. \label{Equation:NonnegativityCone}
\end{align}

It is well-known that deciding nonnegativity is NP-hard even in the polynomial case; see e.g.,~\cite{Laurent:Survey}. Thus, a common way to attack \cref{Intro:KeyProblem}, is to search for \struc{\textit{certificates of nonnegativity}}. These conditions, which imply nonnegativity, are easier to test than nonnegativity itself, and are satisfied for a vast subset of $\NonnegCone$. In the polynomial case, a well-known example of a certificate of nonnegativity are \struc{\textit{sums of squares (SOS)}}, which can be tested via \textit{\struc{semidefinite programming}} \cite{Lasserre:GlobalOpt,Parrilo:Thesis}.
Unfortunately, SOS decompositions do, in general, not preserve the sparsity of $A$.

Another certificate of nonnegativity is a decomposition of $f$ into \textit{\struc{sums of nonnegative circuit functions (SONC)}}, which were introduced by Iliman and the last author for polynomials \cite{Iliman:deWolff:Circuits} generalizing work by Reznick \cite{Reznick:AGI}. Recently, the SONC approach was generalized and reinterpreted by Forsg{\aa}rd and the last author \cite{Forsgaard:deWolff:BoundarySONCcone}. A \textit{\struc{circuit function}} is a function, which is supported on a minimally affine dependent set; see \cref{Definition:Circuit}. For these kind of functions nonnegativity can effectively be decided by solving a system of linear equations; see \cref{Theorem:CircuitsNonnegativity}.

SONCs form a closed convex cone $\struc{\genCS} \subseteq \NonnegCone$. This cone and the functions therein respectively were investigated independently by other authors using a separate terminology. The perspective of considering $\genCS$ as a subclass of nonnegative signomials was originally introduced by Chandrasekaran and Shah \cite{Chandrasekaran:Shah:SAGE-REP} under the name \struc{\textit{SAGE}}, which was later generalized by Chandrasekaran, Murray, and Wiermann \cite{Murray:Chandrasekaran:Wierman:SoncSage, Murray:Chandrasekaran:Wierman:PartialDual}. Furthermore, the notion of SONC was re-interpreted by Katth\"an, Theobald and the third author \cite{Katthaen:Naumann:Theobald:UnifiedFramework} under the name \struc{$\Sc$\textit{-cone}}. We discuss the relation of these different approaches to each other in \cref{Section:Preliminaries}.

\medskip

The key idea of this article is to relax the problem \cref{Intro:KeyProblem} via optimizing over the \textit{\struc{dual SONC cone $\dualCS$}}; see \cref{Definition:DualSONCCone} for a rigorous definition. 
Our approach is motivated by the recent works \cite{Dressler:Naumann:Theobald:SONCdual}, \cite{Murray:Chandrasekaran:Wierman:SoncSage}, and \cite{Katthaen:Naumann:Theobald:UnifiedFramework}, and builds on two key observations, which are the main theoretical contributions:
\begin{enumerate}
	\item The dual SONC cone is contained in the primal one; see \cref{prop:primalindual}.
	\item Optimizing over the dual cone can be carried out by linear programming; see \cref{prop:containment}.
\end{enumerate}

We emphasize that neither the primal nor the dual SONC cone is polyhedral; see in this context also the results in \cite{Forsgaard:deWolff:BoundarySONCcone}. The approach works as follows: First, we investigate a lifted version of the dual cone involving additional linear auxiliary variables (\cref{thm:dualcone} (3)). Second, we show that the coefficients of a given  exponential sum can be interpreted as variables of the dual cone; see \cref{idfct}. Third, we observe that fixing these coefficient variables yields an optimization problem only involving the linear auxiliary variables; see \cref{prop:containment}

Based on our two key observations stated above, we present in \cref{sec:LPapprox} two linear programs \cref{LP-Relax1} and \cref{LP-Relax2} solving a relaxation of \cref{Intro:KeyProblem}.
We implemented the proposed algorithm and provide a collection of examples showing that \cref{LP-Relax1} and \cref{LP-Relax2} work in practice. 
Using the software POEM \cite{poem:software}, we compare our approach exemplarily to existing algorithms for finding SONC and SAGE decompositions via the primal cone $\genCS$, as well as to SOS bounds.

\section*{Acknowledgments}

We thank Thorsten Theobald for his help and input during the development of this article. We thank the anonymous referees for their helpful comments.

TdW is supported by the DFG grant WO 2206/1-1.

\section{Preliminaries}
\label{Section:Preliminaries}

We display vectors in bold notation, e.g., $\struc{\Vector{x}}$ for $(x_1,\ldots,x_n)$.
Throughout the article we write 
$\struc{\R_{> 0}}=\{x\in\R:x > 0\}$. 
Given a set $A \subseteq \R^n$ we denote by $\struc{\conv(A)}$ its \struc{\textit{convex hull}}. We refer to the vertices of $\conv(A)$ as $\struc{\vertices{\conv(A)}}$. 
For a given linear space $L$ we denote by $\struc{\widecheck L}$ its \struc{\textit{dual space}}, and, similarly, for a given cone $C \subseteq \R^n$, we denote by $\struc{\widecheck C}$ its \textit{\struc{dual cone}}.
For the logarithmic function, we use the conventions $0\ln (\frac{0}{y}) = 0$, $\ln (\frac{y}{0}) = \infty$ if $y > 0$ and $\ln (\frac{0}{0}) = 0$, and in addition $\ln (0) = -\infty$.

\subsection{Nonnegativity and the SONC Cone} 

Let $\struc{A} \subseteq \R^n$ be a finite set referred to as the \struc{\textit{support set}}; in what follows we set $\struc{d} = \# A$. 
Recall that we consider exponential sums $f$ of the form \cref{setupgen}. For such an $f$, we set $\struc{\supp(f)} = A$ and denote the vector of coefficients as $\struc{\Vector{c}}$. If $f$ is comprised by a single term, we call it an \struc{\textit{exponential monomial}}.

Following the approach of \textit{fewnomial theory} (also referred to as ``\textit{$A$-philosophy}'' by Gelfand, Kapranov and Zelevinsky; see e.g., \cite{GKZ94}), we fix $A \subseteq \R^n$ and consider the space $\struc{\R^A}$ of all functions with support set $A$, i.e.,
\begin{align*}
	\struc{\R^A} = \spann_\R\left(\left\{\expalpha \with {\alpb}\in A\right\}\right) .
\end{align*}
Since $A$ is fixed, every $f \in \R^A$ can be identified with its coefficient vector and hence there exists a canonical isomorphism $\R^A \simeq \R^d$, i.e., we denote both, vectors and functions, as elements in $\R^A$. 
If $A \subseteq \N^n$, $\R^A$ coincides with the space of real polynomials on the positive orthant supported on $A$.

Recall that the sparse nonnegativity cone $\NonnegCone$ defined in \cref{Equation:NonnegativityCone} is a full-dimensional convex closed cone in $\R^A$.
It is a well-known fact that $f \in \NonnegCone$ only if all coefficients associated to vertices of $\conv(\supp(f))$ are positive; see e.g., \cite{Feliu:Kaihnsa:Yueruek:deWolff} for a detailed proof. Thus, we make the assumption
\begin{align}
	\alpb \in \vertices{\conv(\supp(f))} \ \Rightarrow \ c_{\alpb} > 0. \label{Equation:VerticesAreNonnegative}
\end{align}

Since deciding membership in $\NonnegCone$ is NP-hard, we intend to \textit{certify} membership in $\NonnegCone$ via considering a subcone. For us, the main ingredient is an object called a \textit{circuit function}. Recall that a subset $A'$ of $A$ is called a \struc{\textit{circuit}} if $A$ is minimally affine dependent (i.e., all real subsets of $A'$ are affinely independent); see e.g., \cite{DeLoera:et:al:Triangulations}. 
A special version of circuit functions was first introduced under the name \textit{\struc{simplicial AGI-form}} by Reznick in \cite{Reznick:AGI}, the general definition was given by Iliman and the last author in \cite{Iliman:deWolff:Circuits} focusing on polynomials. Here, we build on a recent, generalized notion by Forsg{\aa}rd and the last author \cite{Forsgaard:deWolff:BoundarySONCcone}.

\begin{definition}[\circuit]
	A function $f\in \R^A$ is called a \struc{\textit{\circuit}} if $\supp(f)$ is a circuit, $\conv(\supp(f))$ is a simplex, and it satisfies \cref{Equation:VerticesAreNonnegative}.
	\label{Definition:Circuit}
\end{definition}

In the special case $A \subseteq \N^n$, {\circuit}s are precisely \struc{\textit{circuit polynomials}} on $\R^{n}_{>0}$ as introduced in \cite{Iliman:deWolff:Circuits}.
Here and in the following we refer to those exponents of $f \in \R^A$ that are contained in the relative interior of $\conv(A)$ by $\struc{\betab} \in A$.

A crucial fact about a \circuit\; $f$ is that its nonnegativity can be decided by an invariant $\struc{\Theta_f}$ called the \struc{\textit{circuit number}} alone. Specifically, Iliman and the last author showed for the case of polynomials, which immediately generalizes to the case of circuit functions:

\begin{theorem}[\cite{Iliman:deWolff:Circuits}, Theorem 1.1]
	Let $f=\sum_{j=0}^r c_{\Vector{\alpha(j)}} \xb^{\Vector{\alpha(j)}} + c_{\betab} \xb^{\betab}$ with $0 \le r \le n$ be a circuit polynomial with $\Vector{\alpha(0)}, \ldots, \Vector{\alpha(r)}\in(2\N)^n$, and let $\Vector{\lambda} \in \R_{> 0}^r$ denote the vector of barycentric coordinates of $\betab$ in terms of the $\Vector{\alpha(0)}, \ldots, \Vector{\alpha(r)}$. Then $f$ is nonnegative if and only if
	\begin{align*}
		|c_{\betab}| \ \leq \ \struc{\Theta_f} \ = \ \prod_{j=0}^r \left(\frac{c_{\Vector{\alpha(j)}}}{\lambda_j}\right)^{\lambda_j}
	\end{align*}
	or if $f$ is a sum of monomial squares.
	\label{Theorem:CircuitsNonnegativity}
\end{theorem}

Note furthermore that a circuit polynomial is nonnegative on $\R^n$ if and only if it is nonnegative on $\R^{n}_{>0}$ (this is, of course, not the case for general polynomials). 
Thus, if one is specifically interested in certifying nonnegativity of polynomials on the entire $\R^n$ using circuit polynomials, then one needs to relax the problem first such that the minimum is attained on $\R^n_{> 0}$. 
We refer readers who are interested in further details to the discussion in \cite[Section 3.1]{Iliman:deWolff:Circuits}.

We consider now the cone of all sums of nonnegative circuits.

\begin{definition}
	We define the \struc{\textit{SONC cone}} $\struc{\genCS}$ as the subset of all $f \in \NonnegCone$, which can be written as a sum of nonnegative circuit functions or nonnegative exponential monomials.
	\label{Definition:SONCCone}
\end{definition}

It is easy to see that $\genCS$ indeed is a convex cone (compare e.g., \cite{Iliman:deWolff:Circuits, Forsgaard:deWolff:BoundarySONCcone}), and it can be shown that $\dim(\NonnegCone) = \dim(\genCS)$; see \cite[Theorem 4.1]{Dressler:Iliman:deWolff:Positivstellensatz} for the non-sparse, polynomial case, which generalizes verbatim to the sparse case considered here.

\medskip 

The SONC cone was studied over the past years by other authors using different approaches and terminology. We especially emphasize two of them:

\begin{enumerate}
	\item  Katth\"an, Theobald, and the third author studied the \struc{$\Sc$-\textit{cone}} in \cite{Katthaen:Naumann:Theobald:UnifiedFramework}. This cone contains sums of nonnegative functions $f: \R^n \to \R \cup \set{\infty}$ of the form
	\begin{align*}
	f(\xb) \ = \ \sum_{\alpb \in \cA} c_{\alpb} |\xb|^{\alpb} +  \sum_{\betab \in \cB} d_{\betab} \xb^{\betab},
	\end{align*}
	where $\cA\subseteq \R^n$ and $\cB\subseteq \N^n\setminus(2\N)^n$ are finite sets of exponents, $\{c_{\alpb}: \alpb\in \cA\}\subseteq \R, \{d_{\betab}: \betab\in \cB\}\subseteq \R$  with either at most one $\betab\in\cB$ such that $d_{\betab}\ne 0$ and $c_{\alpb}\ge0$ for every $\alpb\in\cA$, or $d_{\betab}=0$ for all $\betab\in\cB$ and there exists at most one $\alpb \in\cA$ such that $c_{\alpb} <0$. 
	Since each term with exponent in $\cA$ is isomorphic to an exponential monomial, and it is sufficient to test  nonnegativity of these functions on $\R_{>0}^n$, the functions in the $\Sc$-cone can be regarded as an exponential sum of the form \eqref{setupgen}. Furthermore, one can show that for $\cB=\emptyset$ the $\Sc$-cone coincides with the SONC cone as given in \cref{Definition:SONCCone}. 
	\item Chandrasekaran and Shah introduced an object called \struc{\textit{SAGE cone}} in \cite{Chandrasekaran:Shah:SAGE-REP}, which was then studied further in follow-up articles by Chandrasekaran, Murray, and Wiermann \cite{Murray:Chandrasekaran:Wierman:SoncSage,Murray:Chandrasekaran:Wierman:PartialDual}. This cone contains sums of nonnegative \struc{\textit{AGE functions}}, where an AGE function is of the form
\begin{align*}
	f(\xb) \ = \ \sum_{{\alpb} \in A'} c_{{\alpb}} e^{\langle \xb,{{\alpb}}\rangle } + c_{{\betab}}e^{\langle \xb,{{\betab}}\rangle} \in \R^A,
\end{align*}
such that $A' \subseteq A\subseteq \R^n$, $\betab \in A\setminus A'$, and $c_{{\alpb}} > 0, c_{{\betab}} \in \R$. 

Note that for an {\agf} to be nonnegative, it needs to hold that $\betab \in \conv(A')$.

 The SAGE cone coincides with the SONC cone $\genCS$. This was shown by Reznick in the case of AGI-forms already 1989 in \cite{Reznick:AGI}. AGI-forms are a special case of circuit polynomials when choosing $c_{\alpb}=\lambda_j$ and $c_{\betab}=-1$. For the general case it was first shown (but not explicitly stated) by Wang \cite{Wang:AtMostOneNegativeTerm}. Briefly afterwards, Chandrasekaran, Murray, and Wiermann \cite{Murray:Chandrasekaran:Wierman:SoncSage} finally were the first to explicitly state this fact, which was then observed again in the language of the $\Sc$-cone by Katth\"an, Theobald, and the third author in \cite{Katthaen:Naumann:Theobald:UnifiedFramework}.
\end{enumerate}

\subsection{The Signed SONC Cone}

As a next step, motivated by our approach from optimization, we make a restriction when investigating the SONC cone. 
For a fixed exponential sum $f$, which we intend to minimize, we have additional information on the signs of the coefficients of $f$. 
Since every coefficient corresponds to an element in $A$ due to the isomorphism $\R^{d} \simeq \R^A$ described above, we obtain a decomposition
\begin{align}
	A \ = \ \Mosq\cup \Neg \label{Equation:DecompositionOfSupport}
\end{align}
with disjoint sets $\emptyset\ne \struc{\Mosq} \subseteq \R^n$, corresponding to positive coefficients $c_{\alpb}$, and $\struc{\Neg} \subseteq \R^n$ corresponding to the remaining nonpositive coefficients $c_{\betab}$ in the exponential sum that we consider. 
Thus, we represent exponential sums in this case as
\begin{align}\label{Definition:SignedExponentialSum}
	\struc{f}=\sum_{{\alpb} \in \Mosq} c_{{\alpb}} e^{\langle \xb,{{\alpb}}\rangle }+ \sum_{{\betab} \in \Neg} c_{{\betab}}e^{\langle \xb,{{\betab}}\rangle} \in \R^A.
\end{align}

If we minimize a given function $f$ using the SONC approach, then we restrict to circuits respecting the sign-pattern indicated by $f$. This is the common, tractable approach used by various authors in previous works, e.g., \cite{Dressler:Iliman:deWolff:FirstConstrained, Iliman:deWolff:FirstGP, Murray:Chandrasekaran:Wierman:SoncSage, Murray:Chandrasekaran:Wierman:PartialDual}; it motivates the following definition.

\begin{definition}[Signed SONC cone]\label{Definition:SignedSONCCone}
	Let $A \subseteq \R^n$ be a finite set joint with a decomposition $A = \Mosq \cup \Neg$ in the sense of \cref{Equation:DecompositionOfSupport}. Then the \struc{\textit{signed SONC cone}} $\struc{\CS}$ is the cone of all functions that can be written as a sum of nonnegative {\circuit}s of the form \cref{Definition:SignedExponentialSum} or as nonnegative exponential monomials with support in $\Mosq$. In other words, $\CS$ is the intersection of $\genCS$ with a particular orthant indicated by the pair $(\Mosq,\Neg)$.
	We denote the special case $A^- = \{\betab\}$ as $\struc{\circuitCS}$.
\end{definition}

In fact, by using a generalization of the circuit number and the subsequent notation, we can refine the representation of $\circuitCS$.

\begin{definition} \label{not:Lambda}
	For a non-empty finite set $\Mosq \subseteq \R^n$ and ${\betab} \in \R^n$ let $\Lambda(\Mosq,{\betab})$ be the polytope
	\begin{align}
	\struc{\Lambda(\Mosq, {\betab})} \ = \
	\left\{ \Vector{\lam} \in \R^{\Mosq}_{\ge 0} \quantify \sum_{{\alpb}\in \Mosq} \lam_{{\alpb}} {\alpb}
	= {\betab},\ \sum_{{\alpb}\in\Mosq}\lam_{{\alpb}} = 1\right\}.
	\label{Equation:PolytopeOfConvexCombinations}
	\end{align}
\end{definition}

The polytope $\Lambda(\Mosq, {\betab})$ is nonempty if and only if ${\betab}$ is contained in the convex hull of $\Mosq$ and ${\Lambda(\Mosq, {\betab})}$ consists of a single element whenever the elements in $\Mosq$ are affinely independent.
Particularly, $\Vector{\lambda} \in \Lambda(\Mosq,\betab)$ is, in general, not unique for functions in $\circuitCS$.

Using \cref{Equation:PolytopeOfConvexCombinations}, we may express $\circuitCS$ as follows:

\begin{theorem}[\cite{Katthaen:Naumann:Theobald:UnifiedFramework}, Theorem 2.7] \label{thm:oddImplication} Let $A= \Mosq\cup \{\betab\}$ be defined as in \cref{Equation:DecompositionOfSupport}. The signed SONC cone is the set
	\begin{align*}
	& \circuitCS \ = \ \left\{\sum_{\alpb\in \Mosq} c_{\alpb} \expalpha + 
	c_{\betab} \expbeta \ : \ 
	\begin{array}{c}
		\text{there exists } \Vector{\lam} \in \Lambda(\Mosq, \betab) \text{ such that} \\
	\prod\limits_{\alpb \in \Mosq: \lambda_{\alpb} >0} \left(\frac{c_{\alpb}}{\lambda_{\alpb}}\right)^{\lambda_{\alpb}} \ge -c_{\betab}
	\end{array}
	\right\}.
	\end{align*}
\end{theorem}

Note here that nonnegativity of an AGE function can be certified directly by using $\Theta_f$. There is no need to decompose it into a sum of nonnegative {\circuit}s.

\section{The Dual SONC Cone}

In what follows, we study the dual SONC cone to show containment of the dual in the primal SONC cone (\cref{Sec:DualContainment}) and to obtain a fast linear approximation for global optimization (\cref{sec:LPapprox}).

Due to our goals in this article, we discuss here duality with respect to the signed SONC cone. However, everything generalizes to the full SONC cone immediately.

\subsection{Representations of the Dual SONC Cone}

\begin{definition}[The dual signed SONC cone]\label{Definition:DualSONCCone}
	For an exponential sum $f \in \R^A$ with coefficient vector $\Vector{c}\in \R^A$ we consider the \textit{\struc{natural duality pairing}}
	\begin{align*}
	\struc{\dualvar(f)} \ = \ \sum\limits_{{\alpb}\in \Mosq}\dualdingle_{{\alpb}} c_{{\alpb}} +\sum\limits_{{\betab} \in \Neg} \dualdingle_{\betab} c_{{\betab}} \in \ \R,	
	\end{align*}
	where, as in the primal case, $\dualvar(\cdot) \in \widecheck \R^{A}$ is canonically identified with its (dual) coefficient vector $\struc{\dualvar}$, and hence $\widecheck \R^A \simeq \widecheck \R^d$.
	Using this definition, the \struc{\textit{dual signed SONC cone}} is defined as the set
\begin{align*}
	\struc{\dualCS} \ = \ \left\{\dualvar\in \widecheck \R^A: \dualvar(f)\ge 0 \text{ for all } f\in\CS\right\}.
\end{align*} 
	For brevity, we refer to this cone simply as the \struc{\textit{dual SONC cone}}.
\end{definition}

The following theorem provides two representations of this cone. We need the first one to show containment of the dual SONC cone in the primal one, and the  second representation to obtain the linear program approximating the solution of our global optimization problem \cref{Intro:KeyProblem}.

\begin{theorem}[The dual SONC cone]\label{thm:dualcone}
	Let $A= \Mosq\cup \Neg$ be as in \cref{Equation:DecompositionOfSupport}.
	The following sets are equal.
	\begin{enumerate}
		\item $\dualCS$,
		\item \begin{align*}
			 & \left\{ \dualvar \in \widecheck \R^A :
			 \begin{array}{l}
			 	\text{for all } \; \alpb \in \Mosq, \dualdingle_{{\alpb}} \ge 0 \text{; and for all } \; {\betab} \in \Neg, \\
			 	\text{for all } \; \Vector{\lambda} \in \Lambda(\Mosq, {\betab}) \ , \ \ln (|\dualdingle_{\betab}|) \le \sum\limits_{{\alpb}\in \Mosq} \lam_{{\alpb}} \ln(\dualdingle_{{\alpb}})  
			 \end{array}
			   \right\},
		\end{align*}
		\item \begin{align*}
			\left\{\dualvar \in \widecheck \R^A :
			\begin{array}{l}
				\text{for all } \; {\alpb} \in \Mosq, \dualdingle_{{\alpb}} \geq 0; \text{ and for all } \; {\betab}\in\Neg \\ 
				\text{there exists } \Vector{\tau}\in\R^n, \ln\left(\frac{|\dualdingle_{\betab}|}{\dualdingle_{{\alpb}}}\right)\leq ({\alpb}-{\betab})^T\Vector{\tau} \; \text{ for all } \; {\alpb} \in \Mosq 
			\end{array}
			\right\}.
		\end{align*}
	\end{enumerate}
	\label{thm:dualScone}
\end{theorem}

To prove these representations, we adapt the subsequent theorem from~\cite{Murray:Chandrasekaran:Wierman:SoncSage} to our setting, which basically states that a function in the SONC cone supported on $A= \Mosq\cup \Neg$ can be decomposed into a sum of nonnegative AGE functions supported on $\Mosq\cup \{{\betab}\}$, $\betab \in \Neg$, i.e., the decomposition only uses the support $A$ and there is only one summand per element in $\Neg$.

\begin{theorem}[\cite{Murray:Chandrasekaran:Wierman:SoncSage}, Theorem $2$]\label{mcw:numberofterms}
  Let $f \in \CS$
  with a vector of coefficients $\Vector{c}$. 
  Let $  \Neg \neq \emptyset$.
  Then there exist $\{f^{(\Vector{\beta})} : \Vector{\beta} \in \Neg\} \subseteq \R^A $ with coefficient vectors $\{\Vector{c}^{(\Vector{\beta}) }  : \Vector{\beta} \in \Neg \}$ satisfying
  \begin{enumerate}
    \item $\Vector{c}= \sum_{\Vector{\beta} \in \Neg} \Vector{c}^{(\Vector{\beta})}$ ,
    \item $f^{(\Vector{\beta})} \in \circuitCS$, and 
    \item $c_{\Vector{\alpha}}^{(\Vector{\beta})} = 0$ for all $\Vector{\alpha} \neq \Vector{\beta}$ in $\Neg$.
  \end{enumerate}
\end{theorem}
  
We obtain the following representation of the SONC cone and its dual.

\begin{corollary}\label{cor:mcwrep}~
	Let $A^- \neq \emptyset$. The following statements hold.
	\begin{enumerate}
		\item The SONC cone is the Minkowski sum
		\begin{align*}
			\CS=\sum\limits_{{\betab}\in\Neg} \circuitCS.
		\end{align*}
		\item The dual SONC cone is the set
		\begin{align*}
			\dualCS=\bigcap\limits_{{\betab}\in\Neg} \dualcircuitCS.
		\end{align*}
	\end{enumerate}
\end{corollary}

\begin{proof} The first statement is a direct consequence of~\cref{mcw:numberofterms}.
For the second statement note that Minkowski sum and intersection are dual operations; see, e.g.,  \cite[Theorem 1.6.3]{schneider-book}.
\end{proof}

In particular, this corollary tells us that every nonnegative AGE function is a sum of nonnegative {\circuit}s.

In order to finally prove \cref{thm:dualcone} we need another statement, which essentially combines Lemma 3.6 and a part of the proof of Proposition 3.9  in~\cite{Katthaen:Naumann:Theobald:UnifiedFramework}.

\begin{lemma}[\cite{Katthaen:Naumann:Theobald:UnifiedFramework}]\label{lem:dualAGcone} 
	For ${\betab} \in \Neg \ne \emptyset$ ,
	the dual cone of nonnegative {\circuit}s $\dualcircuitCS$ consists of those $\dualvar \in \widecheck \R^{A}$, where $\dualdingle_{\alpb} \geq 0$ for all ${\alpb} \in \Mosq$, $\dualdingle_{\alpb}=0$ for all $\alpb\in\Neg\setminus\{\betab\}$ and one of the following equivalent conditions hold:
		\begin{enumerate}
		\item $\ln(|\dualdingle_{\betab}|) \leq \sum_{{\alpb}\in \Mosq} \lambda_{{\alpb}} \ln(\dualdingle_{\alpb})$ for all $\Vector{\lam} \in \Lambda(\Mosq, {\betab})$.
		\item There exists $\Vector{\tau}\in\R^n$ such that for all ${\alpb}\in\Mosq \colon \ln\left(\frac{|\dualdingle_{\betab}|}{\dualdingle_{{\alpb} }}\right) \leq ({\alpb} - {\betab})^T \Vector{\tau}.$
		\end{enumerate}
\end{lemma}

\begin{proof}[Proof of \cref{thm:dualcone}]
	In the case $\Neg = \emptyset$, $\dualCS$ only contains sums of nonnegative exponential monomials and the equality of the sets (1)-(3) is clear.
	For $A^- \ne \emptyset$, the statement follows by~\cref{cor:mcwrep} and~\cref{lem:dualAGcone}. Namely, the first representation can be deduced from (1), and  the second one from (2).
\end{proof}

\subsection{The Dual SONC Cone is Contained in the Primal SONC Cone}
\label{Sec:DualContainment}

For $A= \Mosq\cup \Neg$ defined as in \cref{Equation:DecompositionOfSupport}, we identified the dual space of exponential sums supported on $A$ with $\widecheck \R^A$. 
Now we use the reverse identification. 
For every $\dualvar\in \widecheck \R^A$ we associate a function
\begin{align}\label{idfct}
f(\xb)=\sum\limits_{{\alpb}\in \Mosq}\dualdingle_{{\alpb}}\expalpha+ \sum\limits_{{\betab}\in \Neg}\dualdingle_{\betab} \expbeta.
\end{align}
Note that {\circuit}s and {\agf}s are special cases of these functions. 
With this consideration, we identify the dual cone $\dualcircuitCS$ of nonnegative 
{\circuit}s having exponents in
$\Mosq \cup \{{\betab}\}$ with the cone of all functions of the form~\cref{idfct} having coefficients in $\dualcircuitCS$. 
In order to keep notation short, we write $\dualcircuitCS$ for this cone as well. For the cone $\dualCS$ we use the same identification with the notation $\dualCS$.

\begin{proposition}	\label{prop:primalindual}
It holds that
\begin{enumerate}
	\item $\dualcircuitCS\subseteq \circuitCS.$
	\item $\dualCS\subseteq \CS.$
\end{enumerate}
\end{proposition}	
In particular, every function of the form~\cref{idfct} with coefficients in 
$\dualcircuitCS$ or $\dualCS$ is nonnegative.

We point out that \cref{prop:primalindual} was already observed by Katth\"an, Theobald, and the third author in \cite[Remark 3.7]{Katthaen:Naumann:Theobald:UnifiedFramework} without providing a proof.

\begin{proof}~
	\begin{enumerate}
		\item Let $ f \in \dualcircuitCS$ 
		with a corresponding vector of coefficients $\dualvar \in \widecheck \R^A$. 
		By representation (2) of~\cref{thm:dualScone}, we have $\dualdingle_{{\alpb}}\ge 0$ for all ${{\alpb}}\in \Mosq$ and for all 
		$\Vector{\lam} \in \Lambda(\Mosq,{\betab})$ it holds that
		\begin{align*}
			\ln(|\dualdingle_{\betab}|) \le \sum\limits_{{\alpb} \in \Mosq} \lam_{{\alpb}} \ln(\dualdingle_{{\alpb}}) 
			\le \sum\limits_{{\alpb}\in \Mosq, \lam_{{\alpb}} > 0}\lam_{{\alpb}} 
			\ln\left(\frac{\dualdingle_{{\alpb}}}{\lam_{{\alpb}}}\right) 
			= \ln(\Theta_f),
		\end{align*}

		where $\Theta_f$ denotes the circuit number of $f$.
		The last inequality holds as $\lam_{\alpb}\in [0,1]$ for every $\alpb\in \Mosq$ and the logarithmic function is monotonically increasing. 
		Thus, $
		-\dualdingle_{\betab} \le |\dualdingle_{\betab}|\le \Theta_f.$
		Applying~\cref{thm:oddImplication} we obtain the claimed result.
		\item By \cref{Definition:SignedSONCCone}, \cref{Definition:DualSONCCone} and part (1), we obtain
		\begin{align*}
			\dualCS\subseteq \dualcircuitCS\subseteq \circuitCS \subseteq \CS.
		\end{align*}
	\end{enumerate}
\end{proof}

We remark that the reverse implication does not hold in general. 

\begin{example} \label{rmk:dualsmallerprimal}
Consider the function $f(x) \define 1-2e^x+e^{2x}$ with the sets $\Mosq=\{0,2\}$, $\Neg=\{1\}$ and $\dualdingle_0=\dualdingle_2=1,\dualdingle_1=-2$.
As
\begin{align*}
1 = \frac{1}{2} \cdot 0 + \frac{1}{2} \cdot 2 \text{ and }  -\dualdingle_{\betab} = |\dualdingle_{\betab}| = (2^{1/2})^2,
\end{align*}
we have $f \in \CS$. But since
\begin{align*}
\sum\limits_{{\alpb}\in \Mosq}\lam_{{\alpb}} \ln(\dualdingle_{{\alpb}}) 
= 2\left(\frac{1}{2} \ln(1)\right)
= 0
< \ln(2) 
= \ln(|\dualdingle_1|)
\end{align*}
it follows that $f \notin  \dualCS$.
\end{example}

\section{Optimizing Over the Dual SONC Cone via Linear Programming}\label{sec:LPapprox}

In this section, we obtain a computationally fast approximation of the global optimization problem
\begin{align}
	\inf\limits_{\xb \in \R^n} f(\xb)
\label{globalopt}
\end{align}
for exponential sums $f \in \R^A$ and $A= \Mosq\cup \Neg$ defined as in \cref{Equation:DecompositionOfSupport} via the representations of the dual SONC cone in~\cref{thm:dualScone}.

\subsection{Formulation of the Optimization Problem}

First we prove that deciding membership in the dual SONC cone can be done via linear programming.  

\begin{proposition}\label{prop:containment}
	Let 
	\begin{align*}
		f \ = \ \sum\limits_{{\alpb}\in \Mosq} \dualdingle_{\alpb} \expalpha + \sum\limits_{{\betab}\in\Neg}\dualdingle_{\betab} \expbeta
	\end{align*}
	with $\dualvar \in \widecheck \R^A$ and $\dualdingle_{\alpb}\ge 0 \text{ for every } \alpha\in \vertices{\conv(A)}$.
	
	The following linear feasibility program in $\# \Neg$ many variables $(\Vector{\tau}^{(\betab)})_{{\betab}\in\Neg}$ verifies containment in the dual SONC cone. 
	\begin{align}
		&\; \ln\left(\frac{|\dualdingle_{\betab}|}{\dualdingle_{\alpb}}\right) \le ({\alpb}-{\betab})^T\Vector{\tau}^{(\betab)} \;\text{ for all } \; {\betab}\in\Neg,\ {\alpb}\in\Mosq \label{eq:linprg}
	\end{align}
\end{proposition}

\begin{proof}
	The program checks the conditions of~\cref{thm:dualcone}(3). Note that the assumptions ``$\dualdingle_{\alpb}\ge 0 \text{ for every } \alpha\in \vertices{\conv(A)}$'' on $f$ are necessary due to \cref{Equation:VerticesAreNonnegative}.
	As $\dualvar \in \widecheck \R^A$ is fixed, the inequalities are linear and hence \cref{eq:linprg} is a linear program. Moreover, $\dualdingle_{\alpb}\ge 0$ for every $\alpb\in\Mosq$ holds by assumption (or we know trivially that $f$ does not belong to the dual SONC cone). The last inequalities in \cref{thm:dualcone}(3) are satisfied trivially.
\end{proof}

In particular, fixing the non-auxiliary variables $\dualvar$ in a lifted version of the dual cone forms a polyhedron; see \cref{thm:dualcone} and \cref{prop:containment}.

To show that \cref{prop:containment} can be used to obtain an exact linear optimization problem over the dual SONC cone, observe that equivalently to~\cref{globalopt}, we can solve the optimization problem 
\begin{align*}
	\min \left\{ \gamma \ : \ f(\xb) + \gamma \ge 0  \; \text{ for all } \; \xb \in\R^n \right\}.
\end{align*}
Instead of using containment in the SONC cone as a certificate for nonnegativity, i.e., 
solving 
\begin{align*}
	\min \left\{ \gamma \ : \ f(\xb) + \gamma \in \CS\right\},
\end{align*}
we use the dual cone $\dualCS$. Recall that $\dualCS\subseteq \CS$ by \cref{prop:primalindual}. 
In particular, we do not dualize the LP to approximate the solution but optimize $f$ to be a function in the dual cone instead of the primal cone. 
Hence, we compute
\begin{align} \label{eq:dualprob}
\struc{- \widecheck \gamma^*}=	\min \left\{ \widecheck \gamma \ : \ \dualvar+\widecheck \gamma\cdot e_{\Vector{0}}\in \dualCS \right\}, 
\end{align}
where $e_{\Vector{0}}\in\R^A$ is the unit vector corresponding to $e^{\langle\xb,\Vector{0}\rangle}$, i.e., $\struc{\dualdingle_{\Vector{0}}+\widecheck \gamma}$ is the coefficient corresponding to $e^{\langle\xb,\Vector{0}\rangle}$. 

Consider $\dualvar$ to be given via 
\begin{align*}
		f(\xb) + \widecheck \gamma & \ = \ \sum\limits_{{\alpb}\in \Mosq} \dualdingle_{{\alpb}} \expalpha + 	\sum\limits_{{\betab}\in \Neg} \dualdingle_{{\betab}} \expbeta+\widecheck \gamma \\
		 & \ = \ \sum\limits_{{\alpb}\in \Mosq\setminus \{\Vector{0}\}} \dualdingle_{{\alpb}} \expalpha +
		\sum\limits_{{\betab}\in \Neg\setminus \{\Vector{0}\}} \dualdingle_{{\betab}} \expbeta+ (\dualdingle_{\Vector{0}}+\widecheck \gamma).
\end{align*}
Note that the constant term $\dualdingle_{\Vector{0}}$ of $f(\xb)$ can be zero.
By~\cref{thm:dualScone}(3), and assuming $w_{\Vector{0}}:=\widecheck \gamma + \dualdingle_{\Vector{0}}$ and $w_\alpha:=\dualdingle_\alpha$ for all $\alpha\in A\setminus\{\Vector{0}\}$, solving \cref{eq:dualprob} is equivalent to solving
\begin{align}
	\min \left\{\widecheck \gamma: 
	\begin{array}{l}
	\text{for all } \;  {\alpb} \in  \Mosq, w_{{\alpb}} \ge 0 \text{; and for all }\; \betab \in \Neg \\ 
	\text{there exists } \Vector{\tau}\in\R^n,
 \ln\left(\frac{|w_{\betab}|}{w_{{\alpb}}}\right)\leq ({\alpb}-{\betab})^T \Vector{\tau} \text{ for all } \;  {\alpb} \in \Mosq
	\end{array}
 \right\}. \label{dualopt}
\end{align}

Before stating the corresponding optimization program, we emphasize the fact that $\Vector{0}$ is not necessarily contained in $A$, i.e., for the next result we need to include it either in $\Mosq$ or $\Neg$, although we have to determine later to which one of the sets it belongs.

First, we prove several statements addressing this choice.

\begin{lemma}\label{lem:polynomialcase}
	Let $A=\Mosq\cup\Neg\subseteq\R^n$ as in \cref{Equation:DecompositionOfSupport} and $f\in \dualCS$ with $\Vector{0}\in A$. If $f$ is a polynomial, then $\Vector{0}\in\Mosq.$
\end{lemma}

\begin{proof}
	For a polynomial $f$, we have $A\subseteq \N^n$. As $\Vector{0}\in A$, we necessarily have $\Vector{0}\in \vertices{\conv(A)}$. With \cref{Equation:VerticesAreNonnegative} and the fact that $\dualCS\subseteq\CS$, we obtain the statement.
\end{proof}

\begin{lemma}\label{lem:optvalue}
	Let $\Vector{0}\in A=\Mosq\cup\Neg\subseteq\R^n$ as in \cref{Equation:DecompositionOfSupport} and $f\in \dualCS$ with coefficient vector $\dualvar \in \widecheck \R^A$. For the optimal lower bound $-\widecheck \gamma^*\le f(x) $ (as defined in (\ref{eq:dualprob})) and $\struc{c^*}=\ln(|\dualdingle_{\Vector{0}}+\widecheck \gamma^* |)$, we have
		\begin{align}\label{eq:lowerbound}
  	-\widecheck \gamma^* \ = \	\begin{cases}
		& \dualdingle_{\Vector{0}}-e^{c^*} \;  \text{ if } \; \Vector{0}\in\Mosq\\
		& \dualdingle_{\Vector{0}}	+e^{c^*} \; \text{ if }\; \Vector{0}\in\Neg.
		\end{cases}
		\end{align}
\end{lemma}

\begin{proof}
	If $\Vector{0}\in\Mosq$, we have $\dualdingle_{\Vector{0}}+\widecheck \gamma^*\ge 0$ implying $|\dualdingle_{\Vector{0}}+\widecheck \gamma^*|=\dualdingle_{\Vector{0}}+\widecheck \gamma^*$. If $\Vector{0}\in\Neg$, we have $\dualdingle_{\Vector{0}}+\widecheck \gamma^*< 0$ implying $|\dualdingle_{\Vector{0}}+\widecheck \gamma^*|=-\dualdingle_{\Vector{0}}-\widecheck \gamma^*$. This yields the statement.
\end{proof}

From now on, for $A=\Mosq\cup\Neg$ defined as in \cref{Equation:DecompositionOfSupport} and a fixed exponential function 
\begin{align*}
	f \ =\ \sum\limits_{\alpb\in\Mosq\setminus\{\Vector{0}\}} v_{\alpb}\expalpha + \sum\limits_{\betab\in\Neg\setminus\{\Vector{0}\}} v_{\betab}\expbeta + \dualdingle_{\Vector{0}},
\end{align*}
where for all $ \Vector{\alpha}\in \vertices{\conv(A)}$ we have $\dualdingle_{\alpb}\ge 0$ (i.e. $f$ satisfies \cref{Equation:VerticesAreNonnegative}),
with lower bound $-\widecheck \gamma^*$, we consider the following two linear programs in $\# \Neg+1$ variables $(\Vector{\tau}^{(\betab)})_{{\betab}\in\Neg}$ and
$\struc{c}=\ln(|\dualdingle_{\Vector{0}}+\widecheck \gamma|)$. 

	\begin{align*}\label{LP-Relax1}
	& \; \min \; c \tag{$\LPpos$} \\ 
	\st & 
	\begin{array}{ll}
		(1) & \text{for all } \; {\betab}\in\Neg, \text{for all }  {\alpb}\in\Mosq\setminus\{{\Vector{0}}\}:  \ln\left(\frac{|\dualdingle_{\betab}|}{\dualdingle_{\alpb}}\right) \le ({\alpb}-{\betab})^T\Vector{\tau}^{(\betab)},\\ 
		(2) & 
		\ln\left({|\dualdingle_{\betab}|}\right) - c \le (-{\betab})^T\Vector{\tau}^{(\betab)} 	\text{ for all } \; {\betab}\in\Neg ,
	\end{array},
	\end{align*}
	if ${\Vector{0}}\in \Mosq$ and 
	
		\begin{align*}\label{LP-Relax2}
	& \; \min \; c \tag{$\LPneg$} \\ 
	\st & 
	\begin{array}{ll}
	(1) & \text{for all } \; {\betab}\in\Neg\setminus\{{\Vector{0}}\}, \text{for all }  {\alpb}\in\Mosq:  \ln\left(\frac{|\dualdingle_{\betab}|}{\dualdingle_{\alpb}}\right) \le ({\alpb}-{\betab})^T\Vector{\tau}^{(\betab)},\\ 
	(2) & 
c- \ln\left({\dualdingle_{\alpb}}\right) \le {\alpb}^T\Vector{\tau}^{({\Vector{0}})} 	\text{for all } \;{\alpb}\in\Mosq
	\end{array}
	\end{align*}
	if ${\Vector{0}}\in\Neg$.

\begin{lemma}\label{lem:feasibilityLP} Let 
\begin{align*}
	f \ = \ \sum\limits_{\alpb\in\Mosq\setminus\{\Vector{0}\}} v_{\alpb}\expalpha + \sum\limits_{\betab\in\Neg\setminus\{\Vector{0}\}} v_{\betab}\expbeta + \dualdingle_{\Vector{0}}e^{\langle\xb,\Vector{0}\rangle},
\end{align*}
with $\dualdingle_{\Vector{0}}\ne -\widecheck \gamma^*$ and $A=\Mosq\cup\Neg$ defined as in \cref{Equation:DecompositionOfSupport}. 
At least one of the linear programs \cref{LP-Relax1} and \cref{LP-Relax2} has a solution for its corresponding assumption
	\begin{enumerate}
		\item $\Vector{0}\in\Mosq$ or
		\item $\Vector{0}\in\Neg$,
	\end{enumerate}
	 if and only if there exists some $\widecheck\gamma\in\R$ such that  $f+\widecheck\gamma\in\dualCS$.
	 
	For either assumption, the corresponding LP is infeasible if and only if for all $\widecheck\gamma\in\R$ we have $f+\widecheck\gamma\notin\dualCS$.
\end{lemma}

\begin{proof}
Consider $f+\check\gamma^*$. As $\dualdingle_{\Vector{0}}\ne -\widecheck \gamma^*$ we have that $\Vector{0}\in A$. Hence, the inequalities are exactly the inequalities in \cref{thm:dualcone}, except for the fact that we use $c$ instead of $\ln(\dualdingle_{\Vector{0}})$ due to the former substitution. 
\end{proof}

We need to omit $\dualdingle_{\Vector{0}}= -\widecheck \gamma^*$, because in this case the programs (1) and (2)  in \cref{lem:feasibilityLP} are infeasible and unbounded, respectively.
To still obtain a lower bound on the function $f$, one can verify containment in the dual SONC cone by testing feasibility via \cref{eq:linprg}. 
If $f$ is indeed an element in the dual SONC cone, then $0$ is always a lower bound, but not necessarily the optimal bound on $\dualCS$.

From the considerations above and \cref{prop:containment} we can draw the following result.

\begin{theorem}\label{thm:LPAlgo}
	Let 
	\begin{align*}
		 f \ = \ \sum\limits_{\alpb\in\Mosq\setminus\{\Vector{0}\}} v_{\alpb}\expalpha + \sum\limits_{\betab\in\Neg\setminus\{\Vector{0}\}} v_{\betab}\expbeta + \dualdingle_{\Vector{0}} e^{\langle\xb,\Vector{0}\rangle},
	\end{align*}
	with $\dualdingle_{\Vector{0}}\ne -\check \gamma^*$ and $A=\Mosq\cup\Neg$ defined as in \cref{Equation:DecompositionOfSupport} and let $-\check \gamma^*\ne \dualdingle_{\Vector{0}}$ be the optimal value with $f\ge -\gamma*$ as defined in (\ref{eq:dualprob}). The linear programs \cref{LP-Relax1} and \cref{LP-Relax2} solve the optimization problem \cref{dualopt}.
\end{theorem}

\begin{proof}
	We set $A:=A\cup\{\Vector{0}\}$. First, note that we do not know the value of $\dualdingle_{\Vector{0}}+\widecheck \gamma^*$ before computing the optimal value, and particularly we do not know the sign of $\dualdingle_{\Vector{0}}+\widecheck \gamma^*$. Thus, we cannot determine whether $\Vector{0}\in\Mosq$ or $\Vector{0}\in\Neg$ before computing the optimal value.
	
	As we made the assumption $\dualdingle_{\Vector{0}} \ne - \widecheck \gamma^*$, according to \cref{lem:feasibilityLP}, at least one of the problems \cref{LP-Relax1} and \cref{LP-Relax2} is feasible if and only if $f\in\dualCS$. 
	In the case that only one linear program is feasible, $\Vector{0}$ is contained in the corresponding set and hence, this program yields the optimal value. 
	If both programs are feasible, there exist $\widecheck \gamma_1$ and $\widecheck \gamma_2$ such that $\dualdingle_{\Vector{0}}+\widecheck \gamma_1$ is nonnegative and $f + \widecheck \gamma_1 \cdot e^{\langle\xb,\Vector{0}\rangle} \in \dualCS$ for $\Vector{0} \in \Mosq$, and $\dualdingle_{\Vector{0}}+\widecheck \gamma_2$ is negative and $f + \widecheck \gamma_2 \cdot e^{\langle\xb,\Vector{0}\rangle} \in \dualCS$ for $\Vector{0} \in \Neg$.
	
	 Thus, we select the linear program which yields the better bound.
	
	According to \cref{lem:optvalue}, the lower bound on the dual SONC cone is
	\begin{align}
		-\widecheck \gamma^* = \begin{cases}
		& \dualdingle_{\Vector{0}}-e^{c^*} \;  \text{ if } \; \Vector{0}\in\Mosq\\
		& \dualdingle_{\Vector{0}}	+e^{c^*} \; \text{ if }\; \Vector{0}\in\Neg
		\end{cases}.
	\end{align}
\end{proof}

Note that optimizing over the dual cone does not yield the actual optimal value in every case. Consider for example the Motzkin polynomial
\begin{align}\label{ex:motzkin}
f(x,y)=x^2y^4+x^4y^2-3x^2y^2+1.
\end{align}
This is a nonnegative polynomial on $\R^2$ with $\inf_{(x,y)\in\R^2} f(x,y)=0$. Since in the polynomial case we always need $\Vector{0} \in \Mosq$, the linear program \cref{LP-Relax1} for $f$ is the following:
\begin{align*}
 & \; \min \; c\\
\st & \; \ln\left({3}\right) \le 2\tau_2\\
& \; \ln\left({3}\right) \le 2\tau_1\\
& \; \ln\left(3\right) +2\tau_1 + 2\tau_2 \le c,
\end{align*}
returning the lower bound $f\ge -26$ on $\R^2$.

\subsection{Numerical results}

In what follows we present the results of numerical experiments of several examples.

Any LP solver can be used to solve the optimization problem in \cref{thm:LPAlgo}. Here, we used cvxpy \cite{cvxpy,cvxpy_rewriting}; see also \cite{agrawal2019dgp}, in the software POEM \cite{poem:software} available at 
\begin{center}
	\url{http://www.iaa.tu-bs.de/AppliedAlgebra/POEM/} 
\end{center}
on a \verb!Intel(R) Core(TM) i7-8700 CPU! with $3.20$ GHz and $15$ GB of RAM.

To compare our approach with existing results, we restrict our computations to the polynomial case, i.e., the case $A \subseteq \N^n$.
Note that in this setting the convex hull $\conv(\supp(p))$ of exponents of a polynomial $p \in \R^A$ is commonly referred to as the \struc{\textit{Newton polytope}} of $p$.
We use selected examples, mainly from \cite{ExpCompSoncSos18}, to demonstrate our findings. 
Those polynomials that are not explicitly stated in the examples can be found online via 

\begin{center}
	\url{https://www3.math.tu-berlin.de/combi/RAAGConOpt/comparison_paper/}.
\end{center}

The value $\mathsf{opt}$ computed here corresponds to $\check \gamma^*$ as described in \eqref{eq:lowerbound} in the dual case and to a $\gamma^*$ with $p(\xb) \ge - \gamma^*$ in the SOS, SAGE and SONC case.
Hence, a smaller value for $\mathsf{opt}$ means a better lower bound to the polynomial.
To compute this lower bound we need to make a sign change.

In the examples that follow, we denote by ``SAGE'' the bound computed via solving the REP introduced by Chandrasekaran and Shah \cite{Chandrasekaran:Shah:SAGE-REP}, which provides the optimal primal SONC/SAGE bound. With ``SONC'' we denote the covering algorithm for SONC described in \cite[Algorithm 3.4]{ExpCompSoncSos18}. This algorithm solves a GP providing a lower bound for the (optimal) primal SONC/SAGE bound, but (experimentally) with better runtimes and numerical behavior. Thus, it is in particular possible that the bound ``SONC'' is worse than the bound ``Dual SONC''. The bound ``Dual SONC'', however, is always at most as good as ``SAGE'', the optimal primal SONC/SAGE bound (if both bounds can be computed successfully).

The examples are chosen in a way to display that it depends on the particular instance, which approach yields the best bound or has the best runtime respectively.

\begin{example}[\cite{ExpCompSoncSos18}, Example 4.1] \label{ex:1}
	Consider the following polynomial of degree $8$ in two variables with three interior points. 
	\begin{align*}
	p \ = & \ 1 + 3\cdot x_0^{2} x_1^{6} + 2\cdot x_0^{6} x_1^{2} + 6\cdot x_0^{2} x_1^{2} - 1\cdot x_0^{1} x_1^{2} - 2\cdot x_0^{2} x_1^{1} - 3\cdot x_0^{3} x_1^{3}
	\end{align*}
	As expected, the bound returned by our dual approach is worse than the one computed via SONC and SAGE, but it is computed faster; see \cref{table:1}.
	The sum of squares (SOS) approach does not yield a result.
\end{example}
\begin{table}[h]
	\begin{tabular}{|c|c|c|}
		\hline 
		strategy & time & $\mathsf{opt}$ \\ 
		\hline 
		SONC & $0.02254$ & $0.72732$ \\ 
		\hline 
		SAGE & $0.03820$ & $-0.69316$ \\ 
		\hline 
		SOS & $0.30280$ & inf \\
		\hline
		Dual SONC & $0.02706$ & $4.51135$ \\ 
		\hline 
	\end{tabular} 
	\caption{\cref{ex:1}: A polynomial in two variables of degree $8$ with three inner terms.}\label{table:1}
\end{table}
\begin{example}\label{ex:2}
	\begin{align*}
	p = -3  + 1.5\cdot x_1^{6} + 11.5\cdot x_0^{6} - 0.5\cdot x_1^{2} + 0.5\cdot x_0^{4}
	\end{align*}
	In this example all tested approaches yield similar results; see \cref{table:2}.
\end{example}
\begin{table}[h]
	\begin{tabular}{|c|c|c|}
		\hline 
		strategy & time & $\mathsf{opt}$ \\ 
		\hline 
		SONC & $0.01458 $ & $3.11111$ \\ 
		\hline 
		SAGE & $0.01658$ & $3.11111$ \\ 
		\hline 
		SOS & $0.12911$ & $3.11111$\\
		\hline
		Dual SONC & $0.01391$ & $3.28868$ \\ 
		\hline 
	\end{tabular} 
	\caption{\cref{ex:2}: A polynomial in two variables of degree $6$. }\label{table:2}
\end{table}
Since the SONC approach does, in general, not compute the optimal bound of a polynomial on the primal SONC cone, it is also possible that our approach yields better results.
This is demonstrated in the following example.

\begin{example}[\cite{ExpCompSoncSos18}, Example 4.2]\label{ex:3}
	Consider a polynomial whose Newton polytope is a standard simplex with $n = 10$, $d = 30$, and  $200$ terms.
	The bound computed with the dual approach is much better than the one found via SONC. 
	The SAGE approach yields no result, the computations for the SOS approach were aborted after $60$ minutes; see \cref{table:3}. 
\end{example}
\begin{table}[h]
	\begin{tabular}{|c|c|c|}
		\hline 
		strategy & time & $\mathsf{opt}$ \\ 
		\hline 
		SONC & $0.90452$ & $1109.45$ \\ 
		\hline 
		SAGE & $72.90220$ & inf  \\ 
		\hline 
		SOS & $>3600$ & --\footnotemark\\
		\hline
		Dual SONC & $4.21717$ & $-35.25153$ \\ 
		\hline 
	\end{tabular} 
	\caption{\cref{ex:3}: A polynomial in $10$ variables of degree $30$, where $\conv(A)$ is the standard simplex.}\label{table:3}
\end{table}
\begin{example}[\cite{ExpCompSoncSos18}, Example 4.5]\label{ex:4}
	Consider a polynomial with the computationally challenging \struc{\textit{dwarfed cube}}; see \cite{avis1997good}, in dimension $7$ as its Newton polytope.
	For this polynomial, \cref{LP-Relax1} is infeasible.
	The SAGE approach also fails; see \cref{table:4}.
\end{example}
	\begin{table}[h]
		\begin{tabular}{|c|c|c|}
			\hline 
			strategy & time & $\mathsf{opt}$ \\ 
			\hline 
			SONC & $0.34685$ & $-28.2779$ \\ 
			\hline 
			SAGE & $3.19388 $ & inf  \\ 
			\hline 
			SOS & $629.07700$ & $-28.3181$ \\
			\hline
			Dual SONC & -- & inf \\ 
			\hline 
		\end{tabular} 
		\caption{\cref{ex:4}: A polynomial supported on the $7$-dimensional dwarfed cube, scaled by a factor $4$, with $63$ inner terms.}\label{table:4}
	\end{table}
To further illustrate the case of infeasibility in our linear program, consider the following example.
\begin{example}\label{ex:5}
	Consider a polynomial supported on the dwarfed cube in dimension $2$ with two additional interior points.
	\begin{align*}
	p \ = & \ 0.5\cdot x_0^{2} x_1^{4} + 2\cdot x_0^{4} + 1\cdot x_0^{4} x_1^{2} + 2 + 2\cdot x_1^{4} - 1.0\cdot x_0^{1} x_1^{1} -  c \cdot x_0^{3} x_1^{1}.
	\end{align*}
	If we choose $c = 3$, \cref{LP-Relax1} will be infeasible. For $c=1$, however, we get the results presented in \cref{table:5}.
\end{example}
	\begin{table}[h]
		\begin{tabular}{|c|c|c|}
			\hline 
			strategy & time & $\mathsf{opt}$ \\ 
			\hline 
			SONC & $0.02835$ & $-1.58558$ \\ 
			\hline 
			SAGE & $0.02907 $ & $-1.92193$  \\ 
			\hline 
			SOS & $0.08322$ & $-1.92193$ \\
			\hline
			Dual SONC & $0.02486$ & $0.37055$ \\ 
			\hline 
		\end{tabular} 
		\caption{\cref{ex:5}: A polynomial supported on the dwarfed cube in dimension $2$.}
		\label{table:5}
	\end{table}
\footnotetext{Aborted after $60$ minutes.}
\begin{example}\label{ex:6}
	Consider a polynomial supported on \struc{\textit{Kirkman's Icosahedron}}; see \cite{kirkman}, with three additional interior points.
	While the SOS approach was aborted after exceeding a runtime of $60$ minutes, the dual and primal SONC approach and the SAGE approach all yield results very quickly; see \cref{table:6}.
	\begin{align*}
		p \ &= 0.5924068000899325 x_0^{42} x_1^{36} x_2^{36} + 0.9040680744391449 x_0^{6} x_1^{36} x_2^{36} \\
		&+ 0.6286297557636527 x_0^{42} x_1^{12} x_2^{36} + 0.22136661817072706 x_0^{42} x_1^{36} x_2^{12} \\
		&+ 1.9921397074037133 x_0^{6} x_1^{12} x_2^{36} + 2.4444012612478447 x_0^{42} x_1^{12} x_2^{12} \\
		&+ 0.7745809478318744 x_0^{6} x_1^{36} x_2^{12} + 0.4168575879720979 x_0^{6} x_1^{12} x_2^{12} \\
		&+ 2.131772858737973 x_0^{48} x_1^{32} x_2^{24} + 0.5582642102257477 x_1^{32} x_2^{24} \\
		&+ 0.39948625235355123 x_0^{48} x_1^{16} x_2^{24} + 1.055352501861479 x_1^{16} x_2^{24} \\
		&+ 0.5862781645697882 x_0^{24} x_1^{48} x_2^{40} + 0.8297411574785997 x_0^{24} x_2^{40} \\
		&+ 1.5885016970170502 x_0^{24} x_1^{48} x_2^{8} + 0.5937153134426314 x_0^{24} x_2^{8} \\
		&+ 0.7427966893909136 x_0^{36} x_1^{24} x_2^{48} + 0.9341646224001856 x_0^{12} x_1^{24} x_2^{48} \\
		&+ 0.48065798662872594 x_0^{36} x_1^{24} + 0.6729615719188968 x_0^{12} x_1^{24} \\
		&- 1.477600441785058 x_0^{9} x_1^{6} x_2^{1} - 0.1791748172699452 x_0^{3} x_1^{4} x_2^{5} \\
		&- 0.27468070265719946 x_0^{9} x_1^{3} x_2^{7}.
	\end{align*}
\end{example}
	\begin{table}[h]
		\begin{tabular}{|c|c|c|}
			\hline 
			strategy & time & $\mathsf{opt}$ \\ 
			\hline 
			SONC & $4.22899$ & $1.00391$ \\ 
			\hline 
			SAGE & $0.155045$ & $0.50104$  \\ 
			\hline 
			SOS & $>3600$ & --\footnote{Aborted after $60$ minutes.} \\
			\hline
			Dual SONC & $0.15034$ & $2.00542$ \\ 
			\hline 
		\end{tabular} 
		\caption{\cref{ex:6}: A polynomial supported on Kirkman's Icosahedron.}
		\label{table:6}
	\end{table}

\section{Conclusion and Outlook}

The results presented in this paper provide an effective algorithm for optimizing over the dual SONC cone.
Recall that the dual SONC cone is a proper subset of the corresponding primal cone; see \cref{Sec:DualContainment}.
Hence, we observe that, as expected (compare \cref{ex:motzkin}), the linear program developed in \cref{sec:LPapprox} yields, in general, worse results than the SOS, SONC, and SAGE approach.

Since our new approach only relies on solving LPs it is, however, computationally more stable with promising runtimes, and it gives a result whenever a solution in the dual cone exists.

In particular, we obtain an algorithm which yields a bound computed independently of the existing primal SONC, SAGE, and SOS algorithms.

We close the paper by stating two interesting lines of future research.

\subsection{Relaxation of the dual SONC cone}

The constraints guaranteeing containment in the dual cone are very restrictive, occasionally leading to infeasible linear programs \cref{LP-Relax1} and \cref{LP-Relax2}.

In this case, one can solve a relaxed version of the presented linear program, allowing the constraints of \cref{LP-Relax1} and \cref{LP-Relax2} to be violated by some tolerance $\mathsf{tol} \ge 0$ and solving the following optimization problem for a fixed $f \ =\ \sum\limits_{\alpb\in\Mosq\setminus\{\Vector{0}\}} v_{\alpb}\expalpha + \sum\limits_{\betab\in\Neg\setminus\{\Vector{0}\}} v_{\betab}\expbeta + \dualdingle_{\Vector{0}},$
with $\dualdingle_{\alpb}\ge 0$ for all $ \Vector{\alpha}\in \vertices{\conv(A)}$ and lower bound $-\widecheck \gamma^*$. 

\begin{align*}\label{LP-Relaxtol}
			& \; \min \; c+\varepsilon\cdot \mathsf{tol}  \tag{LP-relax}\\ 
			\st & 
			\begin{array}{ll}
				(1) & \text{for all } \; {\betab}\in\Neg\setminus\{{\Vector{0}}\};\\ 
				& \text{for all }  {\alpb}\in\Mosq\setminus\{{\Vector{0}}\}:  \ln\left(\frac{|\dualdingle_{\betab}|}{\dualdingle_{\alpb}}\right) \le ({\alpb}-{\betab})^T\Vector{\tau}^{(\betab)} + \mathsf{tol};\\ 
				(2) & 
				\begin{cases}
					\text{for all } \; {\betab}\in\Neg: \ln\left({|\dualdingle_{\betab}|}\right) \le (-{\betab})^T\Vector{\tau}^{(\betab)} + \mathsf{tol} & \text{if }{\Vector{0}}\in \Mosq, \\
					\text{for all } \;{\alpb}\in\Mosq: \ln\left({\dualdingle_{\alpb}}\right) \le {\alpb}^T\Vector{\tau}^{({\Vector{0}})} + \mathsf{tol} & \text{if }{\Vector{0}}\in\Neg.
				\end{cases}
			\end{array}
		\end{align*}

Here we add $\mathsf{tol}$ as an optimization variable and change the objective function to $c +\varepsilon\cdot \mathsf{tol} $ with choosing $\varepsilon > 0$ as weight of the violation parameter.  Note that this relaxed problem is still a linear program.

This approach yields a solution in a relaxed version of the dual SONC cone.
While there is no guarantee that a solution found in this way is still contained in the nonnegativity cone, one can always find an approximation to a solution this way. 

In fact, since the dual SONC cone is contained in the primal as shown in \cref{sec:LPapprox}, this relaxation also yields a certificate that the found solution is contained in a relaxed version of the primal SONC cone and therefore also in a relaxation of the nonnegativity cone.

\subsection{Primal Polyhedron}

Another possibility for future research would be to investigate the polyhedron we discovered in \cref{prop:containment}. From duality theory we know that there has to exists a primal polyhedron as well. The primal SONC cone itself is, however, not polyhedral; see, e.g., \cite{Forsgaard:deWolff:BoundarySONCcone}.
Hence, it might be interesting to examine the relation of this primal polyhedron to the SONC cone.

\bibliographystyle{amsalpha}
\bibliography{SONC_Dual_ArXiv1}

\end{document}